\documentclass[12pt,reqno]{article}
\usepackage{amsmath, amsfonts, amssymb, amsthm}
\def\floor #1{\lfloor{#1}\rfloor}
\theoremstyle{plain}
\newtheorem{theorem}{Theorem}

\theoremstyle{definition}
\newtheorem*{Ack}{Acknowledgment}
\theoremstyle{remark}

\begin{document}
\title{Sums of Products of Bernoulli numbers of the second kind}

\author{Ming Wu\\
\small Department of Mathematics, JinLing Institute of Technology,\\
\small Nanjing 210001, People's Republic of China\\
\small {\it Email}: mingwu1996@yahoo.com.cn\\
\and Hao Pan\\
\small Department of Mathematics, Shanghai Jiaotong University,\\
\small Shanghai 200240, People's Republic of China\\
\small {\it Email}: haopan79@yahoo.com.cn}
\date{}
\maketitle

\begin{abstract}The Bernoulli numbers $b_0,b_1,b_2,\cdots$ of the second kind are
defined by
$$
\sum_{n=0}^\infty b_nt^n=\frac{t}{\log(1+t)}.
$$
In this paper, we give an explicit formula for the sum
$$
\sum_{\substack{j_1+j_2+\cdots+j_N=n\\
j_1,j_2,\ldots,j_N\geqslant 0}}b_{j_1}b_{j_2}\cdots b_{j_N}.
$$
We also establish a $q$-analogue for
$$
\sum_{k=0}^{n}b_kb_{n-k}=-(n-1)b_n-(n-2)b_{n-1}.
$$
\end{abstract}

The Bernoulli numbers $B_0,B_1,B_2,\cdots$ are defined by
$$
\sum_{n=0}^\infty\frac{B_n}{n!}t^n=\frac{t}{e^t-1}.
$$

Euler (cf. \cite{SD}) noted that if $n>1$ then
\begin{equation}
\label{b1}
\sum_{j=1}^{n-1}\binom{2n}{2j}B_{2j}B_{2n-2j}=-(2n+1)B_{2n}.
\end{equation}

A $q$-analogue of (\ref{b1}) has been given by Satoh in \cite{Sa}.

As a generalization of (\ref{b1}), in \cite{Di} Dilcher proved
that for $n>N/2$ we have
\begin{align}
\label{db1}
&\sum_{\substack{j_1+j_2+\cdots+j_N=n\\j_1,j_2,\ldots,j_N\geqslant 0}}\binom{2n}{2j_1,2j_2,\ldots,2j_N}B_{2j_1}B_{2j_2}\cdots B_{2j_N}\notag\\
=&\frac{(2n)!}{(2n-N)!}\sum_{k=0}^{\floor{(N-1)/2}}c_{k}^{(N)}\frac{B_{2n-2k}}{2n-2k},
\end{align}
where the array $\{c_k^{(N)}\}$ is given by $c_0^{(1)}=1$ and the
recursion
$$
c_{k}^{(N+1)}=-\frac{1}{N}c_{k}^{(N)}+\frac{1}{4}c_{k-1}^{(N-1)}
$$
with $c_k^{(N)}=0$ for $k<0$ and $k>\floor{(N-1)/2}$.

On the other hand, the Bernoulli numbers $b_0,b_1,b_2,\cdots$ of
the second kind are given by
$$
\sum_{n=0}^\infty b_nt^n=\frac{t}{\log(1+t)}.
$$
And we set $b_k=0$ whenever $k<0$. It is easy to check that
\begin{equation}
\label{b2e} \sum_{k=0}^n\frac{(-1)^kb_{n-k}}{k+1}=\delta_{n,0},
\end{equation}
where $\delta_{n,0}=1$ or $0$ according to whether $n=0$ or not.

In \cite{Ho}, Howard used the Bernoulli numbers of the second kind
to give an explicit formula for degenerate Bernoulli numbers. And
some 2-adic congruences of $b_n$ have been investigated by
Adelberg in \cite{Ad}.

In this short note, we shall give an analogue of (\ref{db1}) for
the Bernoulli numbers of the second kind. Define an array
$\{a_{k}^{(N)}(x)\}$ of polynomials by
$$
a_{0}^{(1)}(x)=1, \qquad a_{k}^{(N)}(x)=0\text{ for }k<0\text{ and
}k\geqslant N,
$$
and
$$
a_{k}^{(N)}(x)=-\frac{1}{N-1}((x-N+1)a_{k}^{(N-1)}(x)+(x-N)a_{k-1}^{(N-1)}(x-1))
$$
if $N>k\geqslant 0\text{ and } N\geqslant2$.
\begin{theorem}
For positive integers $n$ and $N$, we have
\begin{equation}
\label{b2} \sum_{\substack{j_1+j_2+\cdots+j_N=n\\
j_1,j_2,\ldots,j_N\geqslant 0}}b_{j_1}b_{j_2}\cdots
b_{j_N}=\sum_{k=0}^{N-1}a_{k}^{(N)}(n)b_{n-k}.
\end{equation}
\end{theorem}
\begin{proof}  Let $s_N(n)$ denote the left-hand side of
(\ref{b2}):
$$
s_N(n)=\sum_{\substack{j_1+j_2+\cdots+j_N=n\\
j_1,j_2,\ldots,j_N\geqslant 0}}b_{j_1}b_{j_2}\cdots b_{j_N}.
$$
Below we use induction on $N$ to show (\ref{b2}).

Clearly $s_1(n)=b_n$, whence (\ref{b2}) holds for $N=1$.

Now let $N>1$ and suppose that (\ref{bth1}) holds for smaller
values of $N$. Note that
$$
\frac{t^{N}}{\log^{N}(1+t)}=\bigg(\sum_{j=0}^\infty
b_jt^j\bigg)^N=\sum_{n=0}^\infty s_N(n)t^n.
$$

For convenience we use $[t^n]f(t)$ to denote the coefficient of
$t^n$ in the power series expansion of $f(t)$. Then
\begin{align*}
s_{N+1}(n)=&[t^n]\frac{t^{N+1}}{\log^{N+1}(1+t)}=[t^{n-N-1}]\frac{1}{\log^{N+1}(1+t)}\\
=&-[t^{n-N-1}]\bigg(\frac{(1+t)}{N}\frac{d}{dt}\bigg(\frac{1}{\log^N(1+t)}\bigg)\bigg)\notag\\
=&-[t^{n-N-1}]\frac{1}{N}\frac{d}{dt}\bigg(\frac{1}{\log^N(1+t)}\bigg)-[t^{n-N-2}]\frac{1}{N}\frac{d}{dt}\bigg(\frac{1}{\log^N(1+t)}\bigg).
\end{align*}
Now
$$
\frac{d}{dt}\bigg(\frac{1}{\log^N(1+t)}\bigg)=\frac{d}{dt}\bigg(\sum_{n=0}^{\infty}s_{N}(n)t^{n-N}\bigg)=\sum_{n=0}^{\infty}(n-N)s_{N}(n)t^{n-N-1}.
$$
Thus by the induction hypothesis on $N$,
\begin{align*}
s_{N+1}(n)=&-\frac{1}{N}\bigg((n-N)s_{N}(n)+(n-N-1)s_{N}(n-1)\bigg)\\
=&-\frac{n-N}{N}\sum_{k=0}^{N-1}a_{k}^{(N)}(n)b_{n-k}-\frac{n-N-1}{N}\sum_{k=0}^{N-1}a_{k}^{(N)}(n-1)b_{n-1-k}\\
=&-\frac{n-N}{N}\sum_{k=0}^{N-1}a_{k}^{(N)}(n)b_{n-k}-\frac{n-N-1}{N}\sum_{k=1}^{N}a_{k-1}^{(N)}(n-1)b_{n-k}\\
=&-\frac{1}{N}\sum_{k=0}^{N}\big((n-N)a_{k}^{(N)}(n)+(n-N-1)a_{k-1}^{(N)}(n-1)\big)b_{n-k}\\
=&\sum_{k=0}^{N}a_k^{(N+1)}(n)b_{n-k}.
\end{align*}
We are done.
\end{proof}

For example, substituting $N=2, 3$ in (\ref{b2}), we obtain that
\begin{equation}
\label{s2}
s_2{(n)}=-(n-1)b_n-(n-2)b_{n-1},
\end{equation}
and
\begin{equation}
s_3{(n)}=\frac{1}{2}(n-1)(n-2)b_n+\frac{1}{2}(n-2)(2n-5)b_{n-1}+\frac{1}{2}(n-3)^2b_{n-2}.
\end{equation}

For arbitrary integer $n$, let
$$
[n]_q=\frac{1-q^{n}}{1-q}.
$$
where $q$ is an indeterminant.

We call $[n]_q$ a $q$-analogue of $n$ since $\lim_{q\to
1}[n]_q=n$. Note that $[1]_q=1$ and $[n-a]_q=[n]_q-q^{n-a}[a]_q$.

Define the $q$-logarithm function by
$$
\log_q(1+t)=\sum_{n=1}^{\infty}\frac{(-1)^{n-1}t^n}{[n]_q}
$$
which is convergent for $|t|<1$.

We also define a $q$-analogue of the Bernoulli numbers of the
second kind by
$$
\sum_{n=0}^\infty b_n(q)t^n=\frac{t}{\log_q(1+t)}.
$$
A $q$-analogue of (\ref{b2e}) is
\begin{equation}
\label{b2r}
\sum_{k=0}^n\frac{(-1)^kb_{n-k}(q)}{[k+1]_q}=\delta_{n,0}.
\end{equation}
Now we give our $q$-analogue of (\ref{s2}).
\begin{theorem} For any integer $n\geqslant 0$, we have
\begin{equation}
\label{qb2}
\sum_{k=0}^{n}q^{k-1}b_k(q)b_{n-k}(q)=-[n-1]_qb_n(q)-[n-2]_qb_{n-1}(q),
\end{equation}
where we set $b_k(q)=0$ for $k<0$.
\end{theorem}
\begin{proof} We use induction on $n$.

When $n=0$, since $[-1]_q=-q^{-1}$ and $b_0(q)=1$ by (\ref{b2r}),
both sides of (\ref{qb2}) coincide with $q^{-1}$.

Now assume that $n>0$ and (\ref{qb2}) holds for smaller values of
$n$. In view of (\ref{b2r}), we have
$$
b_{n-k}(q)=-\sum_{j=1}^{n-k}\frac{(-1)^jb_{n-k-j}(q)}{[j+1]_q}
$$
when $k<n$. Thus
\begin{align*}
&\sum_{k=0}^{n}q^{k-1}b_k(q)b_{n-k}(q)\\
=&q^{n-1}b_n(q)-\sum_{k=0}^{n-1}q^{k-1}b_k(q)\sum_{j=1}^{n-k}\frac{(-1)^jb_{n-k-j}(q)}{[j+1]_q}\\
=&q^{n-1}b_n(q)-\sum_{j=1}^{n}\frac{(-1)^j}{[j+1]_q}\sum_{k=0}^{n-j}q^{k-1}b_k(q)b_{n-k-j}(q)\\
=&q^{n-1}b_n(q)+\sum_{j=1}^{n}\frac{(-1)^j}{[j+1]_q}([n-j-1]_qb_{n-j}(q)+[n-j-2]_qb_{n-j-1}(q))
\end{align*}
where we apply the induction hypothesis in the last step.Observe
that
\begin{align*}
&-\sum_{j=1}^{n}\frac{(-1)^j}{[j+1]_q}\sum_{k=0}^{n-j}q^{k-1}b_k(q)b_{n-k-j}(q)\\
=&\sum_{j=1}^{n}\frac{(-1)^j}{[j+1]_q}([n-j-1]_qb_{n-j}(q)+[n-j-2]_qb_{n-j-1}(q))\\
=&\sum_{j=1}^{n}\frac{(-1)^j}{[j+1]_q}(([n]_q-q^{n-j-1}[j+1]_q)b_{n-j}(q)+([n-1]_q-q^{n-j-2}[j+1]_q)b_{n-j-1}(q))\\
=&[n]_q\sum_{j=1}^{n}\frac{(-1)^jb_{n-j}(q)}{[j+1]_q}+[n-1]_q\sum_{j=1}^{n-1}\frac{(-1)^jb_{n-j-1}(q)}{[j+1]_q}
-\sum_{j=1}^{n}(-1)^jq^{n-j-1}b_{n-j}(q)\\&-\sum_{j=1}^{n-1}(-1)^jq^{n-j-2}b_{n-j-1}(q)\\
=&-[n]_qb_{n}(q)-[n-1]_qb_{n-1}(q)
-\sum_{j=1}^{n}(-1)^jq^{n-j-1}b_{n-j}(q)+\sum_{j=2}^{n}(-1)^jq^{n-j-1}b_{n-j}(q).
\end{align*}
Thus
\begin{align*}
\sum_{k=0}^{n}q^{k-1}b_k(q)b_{n-k}(q)=&q^{n-1}b_n(q)-[n]_qb_{n}(q)-[n-1]_qb_{n-1}(q)+q^{n-2}b_{n-1}(q)\\
=&-[n-1]_qb_{n}(q)-[n-2]_qb_{n-1}(q).
\end{align*}
This completes the proof.
\end{proof}

\begin{Ack} We are grateful to the anonymous referee for his/her
valuable comments on this paper. We also thank our advisor,
Professor Zhi-Wei Sun, for his helpful suggestions on this paper.
\end{Ack}

\medskip
\noindent 2000 {\it Mathematics Subject Classification.} Primary
11B68; Secondary 05A30, 11B65.

\begin{thebibliography}{99}

\bibitem {Ad} A. Adelberg,
\textit{2-adic congruences of N\"olund numbers and of Bernoulli
numbers of the second kind}, J. Number Theory, {\bf 73}(1998),
47-58.

\bibitem {Di} K. Dilcher,
\textit{Sums of products of Bernoulli numbers}, J. Number Theory,
{\bf 60}(1996), 23-41.

\bibitem {Ho} F. T. Howard,
\textit{Explicit formulas for degenerate Bernoulli numbers},
Discrete Math., {\bf 162}(1996), 175-185.

\bibitem {Sa} J. Satoh,
\textit{Sums of products of two $q$-Bernoulli numbers},
J. Number Theory, {\bf 74}(1999), 173-180.

\bibitem {SD} R. Sitaramachandrarao and B. Davis,
\textit{Some identities involving the Riemann zeta-function, II},
Indian J. Pure Appl. Math., {\bf 17}(1986), 1175-1186.
\end{thebibliography}
\end{document}